\documentclass[11pt, oneside, 
a4paper]{article}
\usepackage[ansinew]{inputenc}
\usepackage[english]{babel}
\usepackage{amsbsy,amscd,amsfonts,amsmath,amsopn,amssymb,amstext,amsthm,amsxtra,color,fancybox,float,graphicx,latexsym,srcltx,times,tikz,url}
\newcommand{\nega}{\hspace{-1em}}

\newcommand{\sucenn}[2]{{#1_{1},\hdots,#1_{#2}}}
\newcommand{\sucen}[2]{{#1_{1},#1_{2},\hdots,#1_{#2}}}

\newcommand{\suce}[1]{{#1_{1},#1_{2},#1_{3},#1_{4}}}
\newcommand{\sucet}[1]{{#1_{1},#1_{2},#1_{3}}}

\newcommand{\AAA}{\mathcal{A}}

\newcommand{\colu}[4]{\left[\begin{array}{c}#1\\#2\\#3\\#4
\end{array}\right]}
\newcommand{\colur}[4]{\left[\begin{array}{r}#1\\#2\\#3\\#4
\end{array}\right]}
\newcommand{\colurr}[4]{\left(\begin{array}{r}#1\\#2\\#3\\#4
\end{array}\right)}
\newcommand{\coll}[4]{\left[\begin{array}{r}#1\\#2\\#3\\\vdots\\#4
\end{array}\right]}

\newcommand{\colt}[3]{\left[\begin{array}{r}#1\\#2\\#3
\end{array}\right]}
\newcommand{\colc}[5]{\left[\begin{array}{r}#1\\#2\\#3\\#4\\#5
\end{array}\right]}
\newcommand{\cols}[6]{\left[\begin{array}{r}#1\\#2\\#3\\#4\\#5\\#6
\end{array}\right]}
\newcommand{\cole}[7]{\left[\begin{array}{r}#1\\#2\\#3\\#4\\#5\\#6\\#7
\end{array}\right]}

\newcommand{\m}{\medskip}

\newcommand{\N}{\mathbb{N}}
\newcommand{\Z}{\mathbb{Z}}

\newcommand{\R}{\mathbb{R}}

\newcommand{\TT}{\mathcal{T}}

\newcommand{\tconv}{\operatorname{tconv}}

\newcommand{\dd}{\operatorname{d}}

\newcommand{\card}{\operatorname{card}}

\newcommand{\trop}{\operatorname{trop}}
\newcommand{\rk}{\operatorname{rk}}

\newtheorem{thm}{Theorem}
\newtheorem{lem}[thm]{Lemma}
\newtheorem{dfn}[thm]{Definition}
\newtheorem{cor}[thm]{Corollary}
\newtheorem{ex}[thm]{Example}

\title{Distances on the tropical line determined by two points}

\author{
M.J. de la Puente  
}

\begin{document}
\maketitle

AMS class.: 15A80;  14T05.

Keywords and phrases:  tropical distance, integer length, tropical line,  normal matrix, idempotent matrix,  caterpillar tree, metric graph.

\begin{abstract}
Let $p',q'\in\R^n$. Write $p'\sim q'$ if $p'-q'$ is a multiple of $(1,\ldots,1)$. Two different points $p$ and $q$ in $\R^n/\sim$ uniquely determine a tropical line $L(p,q)$, passing through them, and stable under small perturbations. This line is a balanced unrooted semi--labeled tree on $n$ leaves. It is also a metric graph.
 
If some representatives $p'$ and $q'$ of $p$ and $q$ are the first and second columns of some real normal  idempotent order $n$ matrix $A$, we prove that the tree $L(p,q)$ is described by a matrix $F$, easily obtained from $A$.
We also prove that $L(p,q)$  is caterpillar. We prove that every vertex in $L(p,q)$ belongs to the tropical linear segment joining $p$ and $q$. A vertex, denoted $pq$, closest (w.r.t tropical distance) to $p$  exists in $L(p,q)$.  Same for $q$. The distances between pairs of adjacent vertices in $L(p,q)$ and the distances $\dd(p,pq)$, $\dd(qp,q)$  and $\dd(p,q)$ are certain entries of the matrix $|F|$. In addition,  if $p$ and $q$ are generic, then the tree $L(p,q)$ is trivalent.
The entries of $F$ are differences (i.e., sum of principal diagonal minus sum of secondary diagonal) of order 2 minors of the first two columns of $A$.
 \end{abstract}

\section{Introduction}\label{sec:intro}

Tropical algebra, geometry and analysis are novelties in mathematics.
 As for algebra (also called  extremal algebra, max--algebra, etc.) it is just algebra performed with unusual operations:  $\max$ (for addition) and $+$ (for multiplication). As for geometry, it can be understood as a degeneration (or shadow) of classical algebraic geometry.

\m
Tropical mathematics is an exciting fast growing field of research; see the collective works \cite{Gunawardena,Litvinov_ed,Litvinov_ed_2}, some general references for tropical algebra  \cite{Akian_HB,Baccelli,Butkovic_Libro,Cuninghame_New,Gaubert}, some general references for tropical geometry \cite{Brugalle_fran,Brugalle_engl,Einsiedler,Gathmann,Franceses,Rusos,Mikhalkin_T,Mikhalkin_W,Richter,Speyer_Sturm_2,Viro_log_paper,Viro_basic} and some pioneer works \cite{Cuninghame,Gondran,Wagneur_F,Yoeli,Zimmermann_K} among others.   In \cite{Baker_Faber,Chan} tropical curves are presented as metric graphs.

\m
In classical projective geometry, it is easy to determine the line passing through two different given points $p$ and $q$. If $[\sucen pn]$ and $[\sucen qn]$ are projective coordinates over a field, then the points $x=[\sucen xn]$ in such a line are described by the rank condition
\begin{equation*}
\rk\left[\begin{array}{ccc}
p_1&q_1&x_1\\p_2&q_2&x_2\\\vdots&\vdots&\vdots\\p_n&q_n&x_n\\
\end{array}\right]
=2.
\end{equation*}

\m
A basic question in tropical mathematics is to establish the properties of the unique tropical line $L(p,q)$, stable under small perturbations, determined by two given points $p$ and $q$ (to be precise, $L(p,q)$ is the limit, as $\epsilon$ tends to zero, of the tropical lines going through perturbed points $p^{v_\epsilon}, q^{v_\epsilon}$. Here, $p^{v_\epsilon}$ denotes a translation of $p$ by a
vector  $v_\epsilon$ whose length is $\epsilon$).
The aim of this paper is to answer this question  in a particular instance, namely,  when coordinates of $p$ and $q$ are columns of some normal idempotent square real matrix $A$.

\m
Tropical algebraic varieties can be defined algebraically (by means of ideals) or geometrically (by means of amoebas). Tropical curves can also be defined combinatorially (by means of balanced weighted graphs). For tropical lines,  weights can be disregarded, since they all are equal to one. This paper is about the \emph{combinatorial description} of the line $L(p,q)$. Moreover, we  obtain $L(p,q)$ as a metric graph, with additional information. Indeed, in  metric graphs, leaves have infinite length, while edges have finite length. The point $p$ (which, in general, is not a vertex of $L(p,q)$) sits on a certain leaf of  $L(p,q)$, and we determine the length  from $p$ to the closest inner vertex of $L(p,q)$ (same for $q$). These two lengths are extra information for the metric graph $L(p,q)$. 

\m
In this  paper we never use $-\infty$.
Write $\oplus=\max$ and $\odot=+$. These are the tropical operations addition and multiplication in $\R^n$. Let $(\sucen en)$ denote the canonical  basis in $\R^n$. \label{dfn:canonical} We work in the quotient space $Q^{n-1}:=\R^n/\sim$; see (\ref{eqn:Q}). There is a bijection between $Q^{n-1}$ and $\R^{n-1}$.

\m
Given different $p,q\in Q^{n-1}$, there may exist many tropical lines passing through $p$ and $q$, but there is only one such line which is stable under small perturbations; see \cite{Rusos, Gathmann,  Richter, Tabera_Pap}.
It is denoted $L(p,q)$.

\m
What do we know about tropical lines in $Q^{n-1}$?
The cases  $n=2$ or $3$ are easy.
Set $n=4$.  
In the generic case,
 a \emph{tropical line}  in $Q^3$ is  a \emph{balanced polyhedral complex} consisting of four rays $\suce r$  and an edge $r$, so that
$$L(p,q)=r\cup \bigcup_{j=1}^{4}r_j.$$
The  ray $r_4$ extends infinitely  in the direction of $e_1+e_2+e_3$ and positive sense, and the rays $r_j$ do so in the negative $e_j$ direction, for $j=1,2,3$.

\m
For arbitrary $n$, a \emph{generic} line  $L$ in $Q^{n-1}$ 
is a \emph{balanced unrooted trivalent semi--labeled tree $T$ on leaves marked $1,2,\ldots,n$}. Leaf marked $j$ in $T$  corresponds to ray  $r_j$ in  $L$.  This tree is   \emph{semi--labeled} because  its inner vertices are left unlabeled.
This is all well--known;   see \cite{Gathmann,Franceses,Rusos,Mikhalkin_W,Richter,Sturmfels}.

\m
What do we prove about $L(p,q)$? Let $\tconv(p,q)$ denote the tropical segment joining $p$ and $q$ in $Q^{n-1}$. We have $\tconv (p,q)\subset L(p,q)$, following \cite{Develin_Sturm}.  \emph{Suppose that $p,q$ have representatives in  $\R^n$  whose coordinates are the first and second columns of some normal idempotent square real matrix $A$ of order $n$.}
in this paper we prove that every vertex  of $L(p,q)$ belongs to  $\tconv(p,q)$; see theorem \ref{thm:tree}. This is not true in less restrictive conditions.  Since $\tconv(p,q)$ is compact, then   there is a vertex in $L(p,q)$ closest  to $p$ (same for $q$), with respect to tropical distance (see (\ref{eqn:dist}) for the definition and properties of tropical distance). Moreover, the tree $L(p,q)$ is caterpillar. If $p$ and $q$ are generic, then $L(p,q)$ is trivalent; see also theorem \ref{thm:tree}.

\m
The paper goes as follows. First, we define the \emph{difference} of an order 2 matrix; see definition \ref{dfn:differ}. We define the \emph{matrix of differences} $F=(f_{kl})$ relative to two columns of $A$.  Then,   for $n=4$   we prove that the combinatorics of the tree $L(p,q)$ are determined by the sign of $f_{34}$; see remark in p. \pageref{remark_after_proof}. Moreover,
the tropical distances  $\dd(p,pq)$, $\dd(pq,qp)$, $\dd(q,qp)$ and $\dd(p,q)$ are certain entries of the matrix of absolute values $|F|$. Here $pq$ (resp. $qp$) denotes the vertex of $L(p,q)$ closest to $p$ (resp. to $q$), with respect to tropical distance. Notice that $pq$ and $qp$ are the only vertices of the line $L(p,q)$, for $n=4$. This is theorem \ref{thm:dist4}.
Then, theorem \ref{thm:tree}   is an extension of theorem \ref{thm:dist4} to  arbitrary $n$.

\m
The key to theorem \ref{thm:tree} is additivity  of matrix $F$, as stated in (\ref{eqn:addit}).
To prove that $\dd(p,q)=|f_{12}|$ is straightforward; see lemma \ref{lem:dist_given}.
The proof of theorem \ref{thm:tree} is recursive. It goes as follows. The combinatorics of the tree $L(p,q)$ and the distances between consecutive vertices in it are determined in $n-3$ steps. For each step, we deal with an old tree $T'$ and a new tree $T$. The tree $T$ has one more leaf that $T'$. More precisely, $T$ is a \emph{tropical modification} of $T'$  (see \cite{Brugalle_fran,Brugalle_engl,Mikhalkin_T} for the meaning of modification in tropical geometry). All the distances in $T$ are  kept the same as in $T'$ with one exception: a distance in $T'$ breaks up into two, due to the tropical modification  that has happened. We make this breaking precise by defining \emph{fractures}; see definition \ref{dfn:fracture}. For the understanding of the whole process,  example \ref{ex:2} is provided in full detail, step by step, with accompanying figures \ref{fig:05} to \ref{fig:09}.

\m
We work with only two columns of a normal idempotent matrix (NI, for short). These matrices $A=(a_{ij})$ are  defined by extremely simple linear equalities and inequalities; see (\ref{eqn:NI}).
These inequalities are crucial for us to carry computations through! Normal matrices were first studied by Yoeli (under another name) in \cite{Yoeli}.  Normal idempotent matrices are related to   metrics in \cite{Johns_Kambi_Idempot,Sturm_Yu}. See \cite{Puente_kleene} for   applications of NI matrices to alcoved polytopes, and \cite{Linde_Puente} for applications of normal and NI matrices to tropical commutativity.

\m
Our results and definitions are gathered in  sections \ref{sec:F}, \ref{sec:n=4} and \ref{sec:general}. Lemma \ref{lem:types} and theorem \ref{thm:dist4} were obtained with A. Jim\'{e}nez  and appeared before in \cite{Jimenez_P}. Strictly speaking, the contents of section \ref{sec:n=4}  are included in section \ref{sec:general}. However, we prefer to keep section \ref{sec:n=4} as it stands, because it is helpful for the grasping of the rest of the paper.

\section{Background}\label{sec:back}
For $n\in\N$, set $[n]:=\{1,2,\ldots,n\}$.
Let  $\R^{n\times m}$ denote the set of  real matrices having $n$ rows and $m$ columns. Define tropical sum and product of matrices following the same rules of classical linear algebra, but replacing addition (multiplication) by tropical addition (multiplication).
We will never use classical  multiplication of matrices, in this note.

\m
We will always write the coordinates of points  in columns.



\m
By definition, a square real matrix $A=(a_{ij})$ is \emph{normal} if  $a_{ii}=0$ and $a_{ij}\le 0$, all $i,j\in [n]$. Any real matrix can be normalized, not uniquely;  see \cite{Butkovic_Simple,Butkovic_Libro} for details.
A matrix 
is \emph{idempotent} if  $A=A\odot A$. If each diagonal entry of $A=(a_{ij})$ vanishes, then  $A\le A\odot A$, because   for each $i,j\in[n]$, we have
$$a_{ij}\le \max_{k\in[n]} a_{ik}+a_{kj}=(A\odot A)_{ij}.$$
We will work with  
\emph{normal idempotent matrices} (NI, for short).
Being NI is characterized by the following  linear equalities and
inequalities:
\begin{equation}\label{eqn:NI}
a_{ii}=0,\quad a_{ij}\le0,\quad a_{ik}+a_{kj}\le a_{ij}, \quad i,j,k\in[n], \quad \card\{i,j,k\}\ge2.
\end{equation}
In particular, $a_{ik}+a_{ki}\le0$, for $i,k\in[n]$.


\m
The \emph{tropical determinant} (also called \emph{tropical permanent, max--algebraic permanent}, etc.; see \cite{Butkovic_Libro, Richter}) of   $A=(a_{ij})\in\R^{n\times n}$
 is defined as $$|{A}|_{\trop}
 =\max_{\sigma\in S_n}a_{1\sigma(1)}+a_{2\sigma(2)}+\cdots+a_{n\sigma(n)},$$
where $S_n$ denotes the permutation group in $n$ symbols. The
matrix $A $ is  \emph{tropically singular} if this maximum
 is attained twice, at least. Otherwise, $A$ is \emph{tropically regular}. 
 We will never use 
 classical determinants in this note. See \cite{Develin_Santos_Sturm} for tropical rank issues.

 \m
Two different points $p',q'$  in $\R^n$ determine  the following set of tropical linear combinations:
\begin{equation}\label{eqn:comb_pq}
\{\lambda \odot p'\oplus \mu\odot q'\in\R^n:\lambda,\mu\in\R\}.
\end{equation}
This set is closed under tropical multiplication by any real number $\nu$ i.e., it is closed under classical addition of
vectors $\nu u$, for $u=(1,\ldots,1)$. Therefore, it is useful to   work in the quotient space
\begin{equation}\label{eqn:Q}
Q^{n-1}:=\R^n/\sim
\end{equation}
where $(\sucen an)\sim(\sucen bn)$ if
\begin{equation*}\label{eqn:equiv_rel}
(\sucen an)=\lambda \odot(\sucen bn)=(\sucen{\lambda+b}n),
\end{equation*}
for some $\lambda\in\R$.
The  class of $a=(\sucenn an)\in\R^n$ will be denoted $[\sucenn an]$ or $\overline a$.
The operations $\oplus$ and $\odot$ carry over to  $Q^{n-1}$.
Each element in $Q^{n-1}$ has a \emph{unique representative whose last coordinate is null}; in particular, $Q^{n-1}$ can be identified with the classical hyperplane
\begin{equation}\label{eqn:Hn}
H_n:=\{x\in \R^n: x_n=0\}
\end{equation}
inside $\R^n$. As vector spaces, $H_n$ is isomorphic to $\R^{n-1}$. We will often identify $Q^{n-1}$  with $H_n$ in the sequel. \label{dfn:identify}  By this identification, the topology induced by the tropical distance corresponds to the usual topology.

\m
Given  different points $p,q\in Q^{n-1}$, consider  representatives $p',q'$ in $\R^n$. The image of (\ref{eqn:comb_pq}) in $Q^{n-1}$ is denoted $\tconv(p,q)$ and called the \emph{tropical line segment} determined by $p$ and $q$. \label{dfn:tconv}
By  \cite{Develin_Sturm}, the set $\tconv(p,q)$, viewed in $H_n$, is the concatenation of, at most, $n-1$ ordinary line segments, and the slope of each such line segment is a zero--one vector. \label{dfn:segment} For negative $\lambda$, very large in absolute value, we get
$\lambda \odot p'\oplus \mu \odot q'=\mu\odot q'$, whence  $\lambda \odot p\oplus \mu \odot q=q$ is an endpoint of $\tconv(p,q)$.
(Here we have a difference between classical and tropical mathematics. In classical mathematics, expression (\ref{eqn:comb_pq}) corresponds to a line, not a segment!) The tropical segment $\tconv(p,q)$ is compact and connected, classically.


\m
For $p\in\R^{n}$,  set
$$||p||:=\max_{i,j\in[n]} \{|p_i|, |p_i-p_j|\}.$$
For $p,q\in Q^{n-1}$, choose (unique) representatives $p',q'\in \R^n$ with null last coordinate and set
\begin{equation}\label{eqn:dist}
\dd (p,q):= ||p'-q'||=\max_{i,j\in[n]}\{|p_i-q_i|, |p_i-q_i-p_j+q_j|\}.
\end{equation}
This defines a distance (or a metric) in $(Q^{n-1},\oplus,\odot)$, called \emph{tropical distance}; see \cite{Cohen,Cuninghame_But,Develin_Sturm,Puente_kleene}. We will not use any other distance in this paper.


\m
Recall that the \emph{integer length} (also called lattice length) of a classical segment $ab$ in $\R^n$ joining points $a$ and $b$  is the ratio between the Euclidean length  of $ab$ and the minimal Euclidean length  of integer vectors parallel to $ab$.  If  $a,b\in \Z^2$, then the integer length of $ab$ is one less the number on integer points on the segment $ab$.

\m Recall that the tropical segment $\tconv(p,q)$ is a concatenation of classical bounded segments.
Thus, the integer length  of $\tconv(p,q)$ is the sum of the integer lengths of those segments (see p. \pageref{dfn:tconv}).
It turns out that $\dd(p,q)$ equals the integer length of  $\tconv(p,q)$.

\m
Notice that $\dd$ is additive for tropically collinear points.
For example,  given $p,q,r$ and $s\in Q^2$ (represented in figure \ref{fig:01} by points in $H_3\simeq\R^2$), with $p'=(-2,-2,0)^t$, $q'=(0,0,0)^t$, $r'=(-5,-2,0)^t$ and $s'=(-2,-5,0)^t$,       we have $\dd (p,q)=2$ (not $2\sqrt{2}$!), $\dd (r,s)=\max\{3,6\}=6=3+3$  and
$\dd (r,q)=\max\{5,2,3\}=5=3+2=\dd (s,q)$.

\begin{figure}[h]
 \centering
  \includegraphics[width=8cm,keepaspectratio]{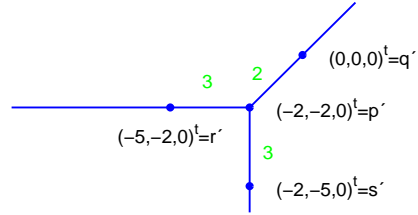}\\
  \caption{Tropical line in $Q^2$ with vertex at the point $p=[-2,-2,0]^t$. It looks like a tripod. Distances are indicated in green.}
  \label{fig:01}
  \end{figure}

\m
For any $S\subseteq [n]$, write $e_S:=\sum_{j\in S}e_j$ and notice that
 \begin{equation}
 \overline{e_S}=-\overline{e_{S^c}}\quad \text{in}\  Q^{n-1},
 \end{equation} where $S^c$ is the complementary to $S$ in $[n]$. In particular, $\overline{e_{12\ldots n}}=\overline{0}$.

\m
Any unbounded closed segment in $\R^{n-1}$ in the direction of some canonical basis vector and negative sense is called a \emph{ray}. Write  $r_j$ for a ray in the $e_j$ direction, for $j\in [n-1]$. Any unbounded closed segment in the direction of $e_{12\ldots n-1}$  and positive sense is also called a \emph{ray}. By abuse of notation,  we denote such a ray by $r_n$.   A ray  $r_j$ is \emph{maximal} inside a line $L$ if the endpoint of $r_j$ is a vertex of $L$. \label{dfn:maximal} An \emph{edge} is a bounded closed segment.


\m
We have $\oplus=\max$ and $\odot=+$. Then,
a \emph{tropical monomial} is a classical linear form $\sum_i a_ix_i$, and a \emph{tropical polynomial} is a maximum
\begin{equation*}
P(\sucen xn)=\max_{a\in \AAA} c_a+a_1x_1+a_2x_2+\cdots a_nx_n,\quad c_a\in\R,
\end{equation*}
and $\AAA\subset \N^n$ finite. The corresponding function $P:\R^n\to \R$  is piecewise linear and concave.   The \emph{tropical hypersurface} determined by $P$ in $\R^n$ is the set of points  where the \emph{maximum is attained twice, at least}. Equivalently, it is the set of points where $P$ is not differentiable; see \cite{Brugalle_fran,Brugalle_engl,Gathmann,Rusos,Mikhalkin_W,Richter,Speyer_Sturm}. In particular, we have tropical lines, planes and  hyperplanes in $\R^n$. Then we mod out by $\sim$, to get tropical lines, planes and  hyperplanes in $Q^{n-1}$.

\m
We work in $(Q^{n-1},\oplus,\odot)$. Algebraically, a tropical line in codimension one (i.e., in $Q^2$) is determined by one  tropical polynomial of degree one. A tropical line in higher codimension is determined by an ideal generated by degree--one tropical polynomials. Tropical lines have been thoroughly studied in \cite{Speyer_Sturm}. The paper \cite{Richter} contains a detailed description of  tropical lines in $Q^3$; see below p. \pageref{lines}.


\m
A generic line $L$ in $Q^2$ looks like a tripod in $H_3\simeq \R^2$; see figure \ref{fig:01}. It consists of three rays $\sucet r$ meeting at vertex. If $L=L(p,q)$, then the vertex  is computed by the \emph{tropical Cramer's rule}; see \cite{Richter,Sturmfels,Tabera_Pap}. It goes as follows: given coordinates $[\sucet p]^t, [\sucet q]^t$ for $p$  and $q$, consider
the   $2\times 2$ tropical minors:
\begin{equation}\label{not:mij}
m_{ij}:=\left|\begin{array}{cc} p_i&q_i\\p_j&q_j\\
\end{array}\right|_{\trop}\nega
=\max\{p_i+q_j,p_j+q_i\}.
\end{equation}
Then the vertex  of $L(p,q)$ is
\begin{equation}\label{eqn:Cramer}
[-m_{23},-m_{13},-m_{12}]^t.
\end{equation}

\m
Fix $n=4$. Let us identify $Q^3$   with $H_4\simeq \R^3$. Set theoretically, a \emph{tropical line} $L$ in $\R^3$ consists of four rays $\suce r$  and, in the generic case, an edge $r$:
\begin{equation*}
L=r\cup \bigcup_{j=1}^{4}r_j.
\end{equation*}
We have $r_j\cap r\neq\emptyset$, for all $j\in[4]$. If $r$ collapses to a point (in the non--generic case), then $r_j\cap r_k\neq\emptyset$, for all $j,k\in[4]$.
 A line $L$  in $Q^3$ belongs to one of the following \emph{combinatorial types}:
\begin{equation*}
\{12,34\},\qquad \{13,24\},\qquad \{14,23\},\qquad \{1234\}.
\end{equation*}
Indeed, the line $L$ is of type $\{ij,kl\}$ if and only if  $L$  has two vertices, denoted $v^{ij}$ and $v^{kl}$, and the segments $r,r_i,r_j$ meet at $v^{ij}$ and  $r,r_k,r_l$ meet at $v^{kl}$, where $\{i,j,k,l\}=[4]$. In particular, types  can be written in various ways: for example,  $\{12,34\}=\{21,34\}=\{21,43\}=\{34,12\}=\{43,12\}$, etc.
 Moreover, the line $L$ is a  trivalent tree if its   type is $\{12,34\},\{13,24\}$ or $\{14,23\}$, and this is the generic case; see figure \ref{fig:02}.
Let $\{i,j,k,l\}=[4]$. We can assume that $i\neq 4\neq j$, without loss of generality. Notice that  \emph{the direction of the segment $r$ of a line $L$ of type $\{ij,kl\}$ is $e_{ij}$}, by the balancing condition.\label{repeat}
On the other hand, if the type of $L$ is
 $\{1234\}$, then the edge $r$ has collapsed to a point,  and the four rays $\suce r$ meet at a point, called vertex of $L$ and denoted $v^{12 34}$. 

\begin{figure}[h]
 \centering
  \includegraphics[width=14cm]{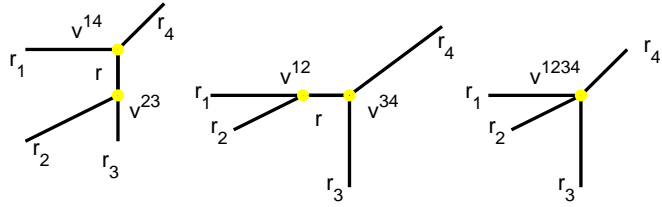}\\
  \caption{Some tropical lines in 3--space: type  $\{14,23\}$ on the left, type $\{12,34\}$ center and  type $\{1234\}$, on the right. These are non--planar balanced polyhedral complexes in $H_4\simeq\R^3$, where the ray $r_4$ points in the direction $e_{123}$,  positive sense. The segment $r$ separates rays $r_1,r_4$ from $r_2,r_3$ in the $\{14,23\}$ case.}
  \label{fig:02}
\end{figure}

\m
It is well--known that two different points $p,q\in Q^3$ determine a unique tropical line $L(p,q)$  passing through them and stable under small perturbations;  see \cite{Develin_Sturm,Richter,Speyer_Sturm}. If $L=L(p,q)$ and we want to compute the vertices of this line,\label{lines} first we must find out the combinatorial type of $L$. Here we follow \cite{Richter}.
For $1\le i<j\le4$, consider  the   $2\times 2$ tropical minors $m_{ij}$ defined in (\ref{not:mij}).
These minors can be arranged into an upper triangular matrix
\begin{equation}\label{eqn:M}
M=\left(\begin{array}{ccc} m_{12}&m_{13}&m_{14}\\&m_{23}&m_{24}\\&&m_{34}\\
\end{array}\right).
\end{equation}
The $m_{ij}$ are not independent: they satisfy the \emph{tropical Pl\"{u}cker relation}, i.e., the following maximum is attained twice, at least:
\begin{equation}\label{eqn:m}
m:=\max\{m_{12}+m_{34}, m_{13}+m_{24}, m_{14}+m_{23}\}.
\end{equation}
Then, by \cite{Richter},
\begin{itemize}
\item the type of $L(p,q)$ is $\{12,34\}$ when $m_{12}+m_{34}<m$,
\item the type of $L(p,q)$ is $\{13,24\}$  when $m_{13}+m_{24}<m$,
\item the type of $L(p,q)$ is $\{14,23\}$  when $m_{14}+m_{23}<m$,
\item the type of $L(p,q)$ is $\{1234\}$  when the maximum  $m$ is attained three times.
\end{itemize}

\m
A point $x$ belongs to $L(p,q)$ if and only if
\begin{equation*}
\rk\left[\begin{array}{ccc}
p_1&q_1&x_1\\p_2&q_2&x_2\\p_3&q_3&x_3\\p_4&q_4&x_4\\
\end{array}\right]_{\trop}\nega
=2;
\end{equation*}
This \emph{tropical rank} condition  means that the value of each of the following $3\times 3$ tropical minors  is attained twice, at least:
\begin{equation*}\label{eqn:omit1}
m_1(x):=\left|\begin{array}{ccc}
p_2&q_2&x_2\\p_3&q_3&x_3\\p_4&q_4&x_4\\
\end{array}\right|_{\trop}\nega
=\max\{x_2+m_{34},x_3+m_{24},x_4+m_{23}\}
\end{equation*}
\begin{equation*}\label{eqn:omit2}
m_2(x):=\left|\begin{array}{ccc}
p_1&q_1&x_1\\p_3&q_3&x_3\\p_4&q_4&x_4\\
\end{array}\right|_{\trop}\nega
=\max\{x_1+m_{34},x_3+m_{14},x_4+m_{13}\}
\end{equation*}
\begin{equation*}\label{eqn:omit3}
m_3(x):=\left|\begin{array}{ccc}
p_1&q_1&x_1\\p_2&q_2&x_2\\p_4&q_4&x_4\\
\end{array}\right|_{\trop}\nega
=\max\{x_1+m_{24},x_2+m_{14},x_4+m_{12}\}
\end{equation*}
\begin{equation*}\label{eqn:omit4}
m_4(x):=\left|\begin{array}{ccc}
p_1&q_1&x_1\\p_2&q_2&x_2\\p_3&q_3&x_3\\
\end{array}\right|_{\trop}\nega
=\max\{x_1+m_{23},x_2+m_{13},x_3+m_{12}\}.
\end{equation*}
Each tropical determinant above has been expanded by the last column, by the tropical Laplace's rule.
Now, for any positive, large enough $u\in\R$,  the points
  \begin{equation*}
y_1(u)=\colur{-u} {-m_{34}}{-m_{24}}{-m_{23}}, y_2(u)=\colur {-m_{34}}{-u} {-m_{14}}{-m_{13}},
y_3(u)=\colur{-m_{24}}{-m_{14}}{-u} {-m_{12}}, y_4(u)=\colur {-m_{23}}{-m_{13}}{-m_{12}} {-u}
\end{equation*}
satisfy that the maximum $m_j(y_j(u))$ is attained three times, for each $j\in[4]$.
Moreover, the point  $y_j(u)$  moves along   a ray $r_j$, as $u $ tends to $+\infty$.

\m
Say the type of $L(p,q)$ is $\{12,34\}$.
Then values $\suce u\in\R$ can be determined so that $y_1(u_1)=y_2(u_2):=v^{12}$ and $y_3(u_3)=y_4(u_4):=v^{34}$, obtaining the following vertices for $L(p,q)$ in $Q^3$:
\begin{equation*}\label{eqn:v^12}
v^{12}=\colur {m_{13}-m_{23}-m_{34}}{-m_{34}}{-m_{24}}{-m_{23}}, \quad
v^{34}=\colur {-m_{24}}{-m_{14}}{m_{13}-m_{12}-m_{14}}{-m_{12}}.
\end{equation*}

Say the type of $L(p,q)$ is $\{13,24\}$.
Similar calculations yield  the following vertices for the line $L(p,q)$, in this case:
\begin{equation}\label{eqn:v^13}
v^{13}=\colur {-m_{24}}{-m_{14}}{ -m_{24}-m_{14}+m_{34}}{-m_{12}},\quad
v^{24}=\colur {-m_{23}}{-m_{13}}{-m_{12}}{ -m_{13}-m_{12}+m_{14}}.
\end{equation}

Say the type of $L(p,q)$ is $\{1234\}$.
Then   we get
\begin{equation*}
v^{1234}=\colur {m_{13}+m_{14}-m_{34}}{m_{12}}{m_{13}}{m_{14}}.
\end{equation*}
 Computations are similar for type $\{14,23\}$. 

\m
Suppose now $n$ that is arbitrary.
A \emph{generic} line $L$ in $Q^{n-1}$ is (identified with)  a \emph{balanced unrooted trivalent semi--labeled tree $T$ on leaves marked $1,2,\ldots,n$} inside $H_n\simeq\R^{n-1}$. Leaf  $j$ of  $T$ corresponds to  ray $r_j$ of the line $L$, while the inner vertices of $T$ are left unlabeled. In particular, generic tropical lines sitting in $Q^{n-1}$ and $Q^{m-1}$ cannot be homeomorphic, if $n\neq m$.

\m
We consider the space  $\TT_n$ of phylogenetic trees, studied in detail in \cite{Billera,Speyer_Sturm} (although this space is denoted $\mathcal{G}'''_{2,n}$  in \cite{Speyer_Sturm}). Then   $\TT_n$ is a simplicial complex of pure dimension equal to $n-4$. The number of facets of $\TT_n$ is
$$(2n-5)!!$$
(i.e., the product of all odd numbers between $2n-5$ and 1, called \emph{Schr\"{o}der number}). \emph{Each facet of $\TT_n$ corresponds to a combinatorial type of unrooted trivalent semi--labeled trees on  $n$ leaves, i.e., to a combinatorial type of generic line in $Q^{n-1}$.} In particular, for $n=4$, there are 3 types  (we have seen these types  above; they were denoted $\{12,34\}$, $\{13,24\}$ and $\{14,23\}$); for $n=5$, there are 15 types; for $n=6$, there are 105 types, and so on.

\m
It is known (see \cite{Sturmfels}) that $\TT_5$ is the Petersen graph: it has 15 edges (these correspond to the 15 types of generic tropical lines in $Q^4$) and 10 vertices (these correspond to types of non--generic tropical lines, where the degree of some vertex of the line is 4). Every generic tropical line in $Q^4$ is a trivalent caterpillar tree on 5 leaves; see \cite{Speyer_Sturm,Sturmfels}.

\m
Recall that a tree is \emph{caterpillar} if it contains  a path passing through every vertex  of degree $\ge2$. For instance, every tree on four leaves is caterpillar.  See figure \ref{fig:03} for trivalent caterpillar and snowflake trees on six leaves.

\m
It is known that $\TT_6$ has 25 vertices, 105 edges and 105 triangles (i.e., there are 105 types of generic tropical lines in $Q^5$): 90 triangles correspond to trivalent caterpillar trees on 6 leaves, and 15 triangles to trivalent snowflake trees on 6 leaves; see
\cite{Richter,Speyer_Sturm}.

\m
Any trivalent semi--labeled tree  $T$ on $n$ leaves  can be described by a finite family of \emph{bipartitions} of $[n]$: a bipartition for each inner edge of $T$. 
\label{dfn:bipartitions}

\m
Given points $p,q\in Q^{n-1}$,  we will  have to describe $L(p,q)$ as a tree, combinatorially. If  $L(p,q)$ is trivalent, this will be achieved by giving a family of bipartitions of $[n]$:
\begin{equation*}
\{S_1,S_1^c\},\{S_2,S_2^c\},\ldots,\{S_t,S_t^c\},
\end{equation*}
for some $t\in\N$ and $S_j\subset[n]$, $j\in [t]$.


\begin{figure}[h]
 \centering
  \includegraphics[width=11cm,keepaspectratio]{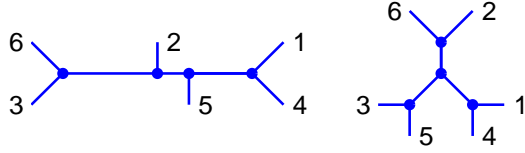}\\
  \caption{Two trivalent semi--labeled trees on six leaves. The inner vertices are not labeled. On the left, caterpillar having three inner edges. This tree is described  the bipartitions $\{36,1245\},\{236,145\},\{2356,14\}$. There is one inner edge separating  leaves marked 3 and 6, from leaves marked 1, 2, 4 and 5.
On the right, a snowflake tree having three inner edges. This tree is described by the bipartitions $\{26,1345\},\{14,2356\},\{35,1246\}$.}
  \label{fig:03}
  \end{figure}

\section{Differences and tropical distances}\label{sec:F}

\begin{dfn}\label{dfn:differ}
Given numbers $a,b,c,d\in\R$, the \emph{difference} of the matrix $\left[\begin{array}{cc}
a&b\\
c&d\\
\end{array}\right]$ is $a+d-b-c$ (principal  diagonal minus secondary diagonal).
\end{dfn}


Consider $A\in\R^{n\times n}$ and write $\underline{i}$ to denote the $i$--th column of $A$. Let $i,j,k,l\in[n]$ with $i<j$  and $k<l$. By $A(kl;ij)$ we denote the minor $\left[\begin{array}{cc}
a_{ki}&a_{kj}\\
a_{li}&a_{lj}\\
\end{array}\right]$.

\begin{dfn}\label{dfn:F}
Fix the $i$--th and  $j$--th   columns of a matrix  $A\in \R^{n\times n}$, with  $1\le i<j\le n$.
For $1\le k<l \le n$, set $F=(f_{kl})$ with
\begin{equation*}
f_{kl}=a_{ki}+a_{lj}-a_{kj}-a_{li}
\end{equation*}
i.e., $f_{kl}$ is the \emph{difference of the minor} $A(kl;ij)$. 
(Obviously, the matrix $F$ depends on $i$ and $j$).
\end{dfn}

\m
Clearly, \begin{equation}\label{eqn:addit}
f_{kl}+f_{lr}=f_{kr}
\end{equation} for $k<l<r$. This \emph{additivity}  (similar to that of Pascal triangle) tells us that $F$ can be recovered from  entries $f_{k-1,k}$. Compare with subadditivity of $A$ shown in (\ref{eqn:NI}).

\begin{lem}\label{lem:f}
If $A\in \R^{n\times n}$ is NI and $F$ is defined above, then
$f_{il}\ge0$, for $i<l$ and $f_{jl}\le0$, for $j<l$. 
\end{lem}
\begin{proof}
$f_{il}=a_{lj}-a_{li}-a_{ij}\ge0$ and $f_{jl}=a_{ji}+a_{lj}-a_{li}\le0$, by (\ref{eqn:NI}).
\end{proof}
Examples of $F$ can be found in p. \pageref{ex:1} and \pageref{eqn:F}.

\m

For $1\le i<j\le n$,
 let $L_{ij}^{A}$ denote the line determined by columns $i$--th and $j$--th of $A$. Write $L_{ij}$, if $A$ is understood.
We will see that \emph{some entries of the  absolute value matrix $|F|$ are equal to some tropical distances  between certain points of $L_{ij}$,  the distance being defined in  (\ref{eqn:dist}).}

\m
To begin with, we have an easy lemma.

\begin{lem}\label{lem:dist_given}
Assume $A\in \R^{n\times n}$ is NI and
fix     $1\le i<j\le n$ with $F$  as in definition \ref{dfn:F}. Then
$\dd(\underline{i},\underline{j})=|f_{ij}|$.
\end{lem}
\begin{proof}
We can assume $i=1$ and $j=2$, by a change of coordinates. Then, by equivalence in $Q^{n-1}$,
\begin{equation*}
\underline{1}-\underline{2}=\coll{-a_{12}}{a_{21}}{a_{31}-a_{32}}{a_{n1}-a_{n2}}=\coll{0}{a_{21}+a_{12}}{a_{31}-a_{32}+a_{12}}{a_{n1}-a_{n2}+a_{12}}.
\end{equation*}
Entries in the last column are non positive, by (\ref{eqn:NI}), the smallest being $a_{21}+a_{12}\le0$, again by (\ref{eqn:NI}). Thus, $\dd(\underline{1},\underline{2})=|a_{21}+a_{12}|=|f_{12}|$.
\end{proof}

\section{Case $n=4$}\label{sec:n=4}

Assume that $i\neq 4\neq j$.  A generic line $L$ is a semi--labeled trivalent tree on four leaves. It has just one inner edge  $r$. Recall  that  $L$ is of type $\{ij,kl\}$ if and only if $e_{ij}$ is the direction of the edge $r$.   Leaves $i,j$ (resp. $k,l$) lie to one endpoint of $r$ (resp. to the other endpoint).

Recall that $L_{ij}^{A}$ denotes the line determined by columns $i$--th and $j$--th of $A$.

\begin{lem}\label{lem:types} Assume $A\in\R^{4\times 4}$ is a NI matrix.
Let   $\{i,j,k,l\}=[4]$ with  $i<j$. Then
 the type of $L_{ij}^{A}$ is
 not $\{ij,kl\}$; it is $\{ik,jl\}$, $\{il,jk\}$  or $\{1234\}$; (easy to remember: $i$ and $j$ must be separated by the comma, unless the type is $\{1234\}$).
\end{lem}

\begin{proof}
Without loss of generality,  assume that $i=1$, $j=2$. Write  $p=\underline{1}$, $q=\underline{2}$ and  $L(p,q)=L_{12}^{A}$. Compute
$M$ in (\ref{eqn:M}) and $m$ in (\ref{eqn:m}), using  (\ref{eqn:NI}), to obtain
\begin{equation}\label{eqn:Mm}
M=\left(\begin{array}{ccc}
0&a_{32}&a_{42}\\&a_{31}&a_{41}\\&&\alpha
\end{array}\right),\  m=\max\{\alpha,a_{32}+a_{41},a_{31}+a_{42}\},\ \alpha=|A(34;12)|_{\trop}.
\end{equation}
 Then, the  value $\alpha$ is attained at the main (resp. secondary) (resp. both) diagonal(s) of $A(34;12)=\left[\begin{array}{cc}
a_{31}&a_{32}\\a_{41}&a_{42}
\end{array}\right]$ if and only if  $\alpha=a_{31}+a_{42}$ (resp. $\alpha=a_{32}+a_{41}$) (resp. $a_{31}+a_{42}=a_{32}+a_{41}$) if and only if the type of  $L_{12}$ is $\{13,24\}$ (resp. $\{14,23\}$) (resp. $\{1234\}$). Thus, $L_{12}$ is not $\{12,34\}$.
\end{proof}

Remark: looking at the former proof and definition \ref{dfn:F}, notice that
 the type of $L_{12}^{A}$ is $\{13,24\}$ if and only if  $f_{34}>0$. If the type were $\{14,23\}$, then $f_{34}<0$ and if the type were $\{1234\}$, then $f_{34}=0$. \label{remark_after_proof}

\m
Recall that maximal rays inside a line were defined in p. \pageref{dfn:maximal}.

\begin{lem}\label{lem:nuevo}
Assume $A\in \R^{4\times 4}$ is NI and
let   $\{i,j,k,l\}=[4]$ with  $i<j$. Then the vertices of the line $L_{ij}^{A}$ belong to the tropical segment $\tconv(\underline{i},\underline{j})$. Moreover, $\underline{i}\in r_j$ and $\underline{j}\in r_i$, where $r_i,r_j$ are maximal rays in $L_{ij}^A$.
\end{lem}

\begin{proof}
Without loss of generality,  assume that $i=1$, $j=2$. The points \underline{1} and \underline{2} have coordinates
$$\colu {0}{a_{21}}{a_{31}}{a_{41}},\quad \colu {a_{12}}{0}{a_{32}}{a_{42}},$$
respectively and we know that the coordinates of the vertices of  $L_{12}$  depend on the type of $L_{12}$. This type is not $\{12,34\}$, by lemma \ref{lem:types}.

Say the type of $L_{12}$ is $\{13,24\}$. Then $M$, $m$ and $\alpha$ are shown in (\ref{eqn:Mm}), with
\begin{equation}\label{eqn:alpha}
a_{32}+a_{41}<a_{31}+a_{42}=\alpha.
\end{equation}
Using  (\ref{eqn:v^13}), the vertices of  $L_{12}$ are
\begin{equation}\label{eqn:v^13,v^24}
v^{13}
=\colur{-a_{41}}{-a_{42}}{a_{31}-a_{41}}{0},
\quad v^{24}
=\colur{-a_{31}}{-a_{32}} {0}{a_{42}-a_{32}}.
\end{equation}
We have
\begin{equation*}
v^{13}=\colurr{-a_{41}}{a_{21}-a_{41}}{a_{31}-a_{41}}{0}\oplus \colurr{a_{12}-a_{42}}{-a_{42}}{a_{32}-a_{42}}{0}=\overline{\underline{1}\odot(-a_{41})\oplus\underline{2}\odot(-a_{42})}
\end{equation*}
and
\begin{equation*}
 v^{24}=\colurr{-a_{31}}{a_{21}-a_{31}} {0}{a_{41}-a_{31}}\oplus\colurr{a_{12}-a_{32}}{-a_{32}} {0}{a_{42}-a_{32}}=\overline{\underline{1}\odot(-a_{31})\oplus\underline{2}\odot(-a_{32})},
\end{equation*}
using  inequalities (\ref{eqn:NI}) and (\ref{eqn:alpha}). This shows that $v^{13}$ and $v^{24}$ belong to $\tconv(\underline{1},\underline{2})$.
Moreover
\begin{equation}\label{eqn:1menosv13}
\underline{1}-v^{13}=\colur {a_{41}}{a_{21}+a_{42}}{a_{41}}{a_{41}}=\colur {0}{a_{21}+a_{42}-a_{41}}{0}{0}
=\colur{0}{f_{24}}{0}{0},
\end{equation}
whence  $\underline{1}\in r_2$. Similarly, $\underline{2}-v^{24}=[-f_{13},0,0,0]^t$, whence  $\underline{2}\in r_1$.

\m
Computations are analogous  if the type of line $L_{12}$ is $\{14,23\}$.
\end{proof}

Recall that the tropical distance induces the usual topology. By compactness  of $\tconv(\underline{i},\underline{j})$, there is a vertex in $L_{ij}^{A}$ \emph{closest to $\underline{i}$}, denoted
 $\underline{ij}$, and  a vertex in $L_{ij}^{A}$ \emph{closest to $\underline{j}$},  denoted $\underline{ji}$,  distances considered   tropically.
Of course, $\underline{ji}=\underline{ij}$ if and only if  $L_{ij}$ is $\{1234\}$.

\m
In the following theorem, notice that distances depend on the type of $L_{ij}^{A}$.

\begin{thm}\label{thm:dist4}
Assume $A\in \R^{4\times 4}$ is NI and
let   $\{i,j,k,l\}=[4]$ with  $i<j$.
If the type of the line $L_{ij}^{A}$ is $\{ik,jl\}$, then
\begin{enumerate}
\item $\dd (\underline{i}, \underline{ij})=
|f_{jl}|$,
\item $\dd (\underline{j}, \underline{ji})=
|f_{ik}|$,
\item $\dd (\underline{ij}, \underline{ji})=|f_{kl}|$ (this case is easy to remember).
\end{enumerate}
\end{thm}

\begin{proof}
Without loss of generality,  assume that $i=1$, $j=2$. We know that the  type of $L_{12}$ is not $\{12,34\}$, by lemma \ref{lem:types}.

Say the type of $L_{12}$ is $\{13,24\}$, so that $k=3$, $l=4$.   By definition of $F$ and (\ref{eqn:alpha}), we have $f_{34}>0$.  Go back to (\ref{eqn:v^13,v^24}), where coordinates for $v^{13}$ and $v^{24}$ were computed, to get
\begin{equation*}
v^{13}-v^{24}=\colur{a_{31}-a_{41}}{a_{32}-a_{42}}{a_{31}-a_{41}}{a_{32}-a_{42}}
=\colur{a_{31}-a_{41}-a_{32}+a_{42}}{0}{a_{31}-a_{41}-a_{32}+a_{42}}{0}
=\colur{f_{34}}{0}{f_{34}}{0}
\end{equation*}
and we obtain
\begin{equation*}
\dd(v^{13},v^{24})=f_{34}.
\end{equation*} Moreover, from (\ref{eqn:1menosv13}) and lemma \ref{lem:f} (for $j=2$),  we get
\begin{equation*}\label{eqn:d1v^13}
\dd(\underline{1},v^{13})=-f_{24}=|f_{24}|,
\end{equation*}
Similarly,
\begin{equation*}
\dd(\underline{2},v^{24})=f_{13}=|f_{13}|.
\end{equation*}
Now
 \begin{equation*}
\underline{2}-v^{13}=\colur{a_{12}+a_{41}}{a_{42}}{a_{32}+a_{41}-a_{31}}{a_{42}}=\colur{a_{12}+a_{41}-a_{42}}{0}{a_{32}+a_{41}-a_{31}-a_{42}}{0}
=\colur{-f_{14}}{0}{-f_{34}}{0}.
\end{equation*}
By additivity (\ref{eqn:addit}), we have $f_{13}+f_{34}=f_{14}$,  with $f_{13}\ge0$, $f_{14}\ge0$ and $f_{34}>0$. Thus, by the definition of tropical distance, we get
\begin{equation*}
\dd(\underline{2},v^{13})=\max\{f_{14}, f_{34}, f_{13}\}=f_{14}.
\end{equation*}
We have $\dd(\underline{2},v^{24})=f_{13}<f_{14}=\dd(\underline{2},v^{13})$,
showing that $v^{24}$ is closer to $\underline{2}$ than $v^{13}$. Thus we can relabel as follows
\begin{equation*}
v^{24}=\underline{21},\qquad v^{13}=\underline{12}.
\end{equation*}
This proves the three statements for type $\{13,24\}$.
Computations are similar  if the type of $L_{12}$ is $\{14,23\}$.
\end{proof}

\begin{ex}\label{ex:1}
Assume that $*\in\R$ are such that $A$ is NI, with  $$A=\left[\begin{array}{rrcc}
0&-12&*&*\\
-10&0&*&*\\
-11&-14&0&*\\
-15&-13&*&0
\end{array}\right],$$
 (this can be achieved, for instance, taking   $-20\le a_{kl}\le -10$,  for $k,l=3,4$ and $k\neq l$). We have
\begin{equation*}
F=\left(\begin{array}{rrr}
22&9&14\\
&-13&-8\\
&&5
\end{array}\right)
\end{equation*}
and $\dd(\underline{1},\underline{2})=22$, by lemma \ref{lem:dist_given}.
By  the last part in theorem \ref{thm:dist4}, we get
 \begin{equation*}
 \dd(\underline{12},\underline{21})=|f_{34}|=5\neq0,
 \end{equation*}
 whence the type of $L_{12}$ is not $\{1234\}$. It can be either $\{13,24\}$ or $\{14,23\}$, since 1 and 2 must be separated by the comma,  by lemma \ref{lem:types}. We have
 \begin{equation*}
 \underline{1}-\underline{2}=\colur{12}{-10}{3}{-2}=\colur{14}{-8}{5}{0},\quad \dd(\underline{1},\underline{2})=22.
 \end{equation*}
 If the type were $\{14,23\}$, by theorem \ref{thm:dist4} we would have
 \begin{equation*}
 \dd(\underline{1},\underline{12})=|f_{23}|=13,\quad \dd(\underline{2},\underline{21})=|f_{14}|=14,\quad 22\neq 13+5+14,
 \end{equation*}
 contradicting that the tropical distance is additive for three tropically collinear points.
 Thus the type is $\{13,24\}$ and then
 \begin{equation*}
 \dd(\underline{1},\underline{12})=|f_{24}|=8,\quad \dd(\underline{2},\underline{21})=|f_{13}|=9,\quad 22= 8+5+9.
 \end{equation*}

A longer way to obtain the same result is computing $M, m$ and $\alpha$ in (\ref{eqn:Mm}). We get  that the type of $L_{12}^A$ is   $\{13,24\}$, and then formulae (\ref{eqn:v^13}) provide the coordinates of \underline{12} and \underline{21}.
\end{ex}


\begin{cor}\label{cor:adiffs} Assume $A\in \R^{4\times 4}$ is NI and
let     $1\le i<j\le 4$.
If the type of the line $L_{ij}^A$ is $\{1234\}$, then for $k\in[4]\setminus \{i,j\}$ we have
\begin{enumerate}
\item $\dd (\underline{i}, \underline{ij})=
|f_{jk}|$,
\item $\dd (\underline{j}, \underline{ji})=
|f_{ik}|$. \qed
\end{enumerate}
\end{cor}

\section{General case}\label{sec:general}
Our aim for this section is to describe the tree $L(p,q)$ through the matrix $F$.
Let $pq$ (resp. $qp$) denote the vertex  closest to $p$ (resp $q$) in $L(p,q)$, if such a vertex  exists. These two are the only inner vertices of the line $L(p,q)$ that we will \emph{consistently label}. Vertices of $L(p,q)$ may receive temporary labels, such as $v,w,x,y,z$ etc.

\m
Let $A\in \R^{n\times n}$ be a NI matrix. For the rest of the paper, we assume that $p=\underline{1}$ and $q=\underline{2}$, so that $L(p,q)=L(\underline{1},\underline{2})=L_{12}^A$. This is no loss of generality. If $F$ is as in definition \ref{dfn:F}, then
\begin{equation}\label{eqn:pos_negat}
f_{1k}\ge0, \qquad f_{2k}\le0,\qquad \forall k
\end{equation}
\begin{equation}\label{eqn:max}
f_{12}=\max_{1\le k< l\le n} |f_{kl}|
\end{equation}
 by lemma \ref{lem:f} and the NI condition (\ref{eqn:NI}).


\m

\label{not:^s} Notation: For  $3\le s\le n$, let $A^s$ (resp. $F^s$) denote the principal  minor of $A$ (resp. of $F$) of order $s$; in particular, $A^n=A$. The first two columns of $A^s$  are denoted $\underline{1}^s$ and $\underline{2}^s$. The   line $L(\underline{1}^s,\underline{2}^s)$  is denoted $L^s$. It sits inside $Q^{s-1}$, which can be identified with $H_s\simeq \R^{s-1}$. In particular, $L^n=L(p,q)$.
Let $\underline{12}^s$ (resp. $\underline{21}^s$) denote the vertex of $L^s$ closest to $\underline{1}^s$ (resp. to $\underline{2}^s$), if such a vertex  exists. Let $r_j^s$ denote any ray in the $e_j$ negative sense inside  $\R^{s-1}$,  for $j\in[s-1]$, and  $r_s^s$ any ray in the $e_{12\ldots s-1}$ positive sense.
We know that $L^s$ is the finite union of $s$ rays $\sucenn {r^s}s$ and some edges $\sucenn ht$, for certain  $t\in \N\cup\{0\}$.

\begin{dfn}\label{dfn:active}
Fix $s$ with $3\le s\le n$. If, for some $1\le k< l\le s$,  $|f_{kl}|$ equals either the distance between two adjacent vertices in $L^s$  or it equals $\dd(\underline{1}^s,\underline{12}^s)$ or  $\dd(\underline{21}^s,\underline{2}^s)$, then we will say that $f_{kl}$ is $s$--\emph{active}.
\end{dfn}

\begin{dfn}\label{dfn:fracture}
If $|a|=|b|+|c|$ with non--zero $a,b,c\in\R$, we say that $a$ \emph{fractures by means of $b$.} We also say that $a$ was \emph{formerly active} and that  $b,c$ are \emph{newly active.} 
\end{dfn}

\m
Consider the matrix $F^s$ and assume that  $f_{kl}$ is $(s-1)$--active, with $1\le k<l\le s-1$. Then,  $f_{kl}$ fractures by means of some entry of the $s$--th column, if and only if
\begin{equation}\label{eqn:ineq}
|f_{kl}|>|f_{ks}|.
\end{equation} Indeed,  we will have $|f_{kl}|=|f_{ks}|+|f_{ls}|$, following from  additivity (\ref{eqn:addit}). In practice, to find out if a fracture occurs by means of some entry of the $s$--th column, we can minimize the absolute value of the entries of the $s$--th column of $F^s$. 

\begin{lem}\label{lem:crossed}
Let $A\in\R^{n\times n}$ be NI and $3\le s\le n$. 
Then point $\underline{2}^s$ lies to the northwest of $\underline{1}^s$ inside $H_n\subset \R^{n}$.
\end{lem}

\begin{proof}
By equivalence in $Q^{n-1}$, the coordinates of $\underline{1}^s$ and $\underline{2}^s$ in $H_n$ are
\begin{equation*}
\cols{-a_{s1}}{a_{21}-a_{s1}}{a_{31}-a_{s1}}{\vdots}{a_{s-1,1}-a_{s1}}{0},\quad
\cols{a_{12}-a_{s2}}{-a_{s2}}{a_{32}-a_{s2}}{\vdots}{a_{s-1,2}-a_{s2}}{0}
\end{equation*}
where the  first and second coordinates compare as follows:
 $$-a_{s1}\ge a_{12}-a_{s2},$$
 $$a_{21}-a_{s1}\le -a_{s2},$$ by (\ref{eqn:NI}). This implies the result. 
\end{proof}

\begin{thm}\label{thm:tree}
Let $n\ge3$ and assume $p,q$ are different points in $Q^{n-1}$ having representatives $p',q'$ in $\R^n$ whose coordinates are the first and second columns of a NI matrix $A\in\R^{n\times n}$. 
Then the matrix $F$, as in   definition \ref{dfn:F}, describes  the line $L(p,q)$ as a balanced unrooted  semi--labeled tree on  $n$ leaves, which is caterpillar.   Every vertex in $L(p,q)$ belongs to $\tconv(p,q)$. The vertices $pq$ and $qp$ exist in $L(p,q)$. The distances between pairs of adjacent vertices in $L(p,q)$ and the distances $\dd(p,pq)$, $\dd(qp,q)$  and $\dd(p,q)$ are certain entries of the matrix $|F|$. In addition,  if $p$ and $q$ are generic, then $L(p,q)$ is trivalent.
\end{thm}

\begin{proof}
We have $p=\underline{1}$ and $q=\underline{2}$ and $\dd(\underline{1},\underline{2})=f_{12}$, by lemma \ref{lem:dist_given} and  (\ref{eqn:pos_negat}). Write $L=L(\underline{1},\underline{2})=L_{12}^A$.

\m
First, let us assume that the couple $p, q$ is \emph{generic.}  Then, $L$ and $F$ are also generic.

With notation from p. \pageref{not:^s}, let us begin with the line  $L^2$, joining the points
$\left[\begin{array}{cc}
0\\
a_{21}
\end{array}\right]=\left[\begin{array}{cc}
-a_{21}\\
0
\end{array}\right]$ and
$\left[\begin{array}{cc}
a_{12}\\0
\end{array}\right]$.
Then
$$\dd(\underline{1}^2, \underline{2}^2)=|a_{12}+a_{21}|=f_{12},$$
by lemma \ref{lem:dist_given}. We have $f_{12}\neq0$, by genericity of $F$ and $f_{12}$ is 2--active. This is the initial step.

\m
The proof proceeds by recursion,  for $3\le s\le n$. In the $s$-- th step,
the  line $L^s$ is obtained from the line  $L^{s-1}$, by tropical modification. This precisely means that  exactly one $(s-1)$--active entry of $F^{s-1}$ fractures.
Moreover, after  the $s$--th step is completed, we have the following properties:
\begin{enumerate}  \label{prop:list}
\item in each row of $F^s$, there is some $s$--active entry,\label{prop:some_in_row}
\item there are exactly two $s$--active entries in the last column of $F^s$; these are newly active, \label{prop:two_in_last}
\item there are some negative  and some positive $s$--active entries in $F^s$,\label{prop:signs}
\item the sum of the absolute values of all $s$--active entries  in $F^s$ is equal to $f_{12}$.\label{prop:sum}
\end{enumerate}

\begin{itemize}\label{eqn:Cramer2}
\item if $s=3$, then $f_{12}+f_{23}=f_{13}$, by additivity (\ref{eqn:addit}). By (\ref{eqn:pos_negat}) and (\ref{eqn:max}),
$$|f_{12}|=|f_{13}|+|f_{23}|$$ is a fracture of $\dd(\underline{1}^2,\underline{2}^2)=f_{12}$. The line $L^3$ has a vertex, which we denote $w^3$, whose coordinates are given in (\ref{eqn:Cramer})
\begin{equation*}
w^3=\colt{-m_{23}}{-m_{13}}{-m_{12}}=\colt{-a_{31}}{-a_{32}}{0}=\overline{\underline{1}\odot(-a_{31})\oplus\underline{2}\odot(-a_{32})},
\end{equation*}
 equalities holding by the NI hypothesis.  Then
\begin{equation}\label{eqn:norte_este}
\underline{1}^3-w^3=\colt{a_{31}}{a_{32}+a_{21}}{a_{31}}=\colt{0}{f_{23}}{0}, \quad\underline{2}^3-w^3=\colt{a_{31}+a_{12}}{a_{32}}{a_{32}}=\colt{-f_{13}}{0}{0}
\end{equation}
whence
\begin{equation*}
\dd(\underline{1}^3,w^3)=|f_{23}|=-f_{23},\quad \dd(\underline{2}^3,w^3)=|f_{13}|=f_{13}.
\end{equation*}
Now $f_{13},f_{23}$ become 3--active, while $f_{12}$ stops being active.

Equalities (\ref{eqn:norte_este}) tell us that walking northbound from point $\underline{1}^3$ for $|f_{23}|$ units, we reach $w^3$, and walking eastbound from point $\underline{2}^3$ for $f_{13}$ units, we also reach $w^3$; see figure \ref{fig:05}, left. The line $L^3$ satisfies the statement of the theorem and it is trivalent.

\item if $s=4$, there are two cases: either $f_{34}<0$ or $f_{34}>0$ (by genericity of $F$, we have $f_{34}\neq0$). Both cases were studied in theorem \ref{thm:dist4}. Being generic, the tree $L^4$ is  of type $\{13,24\}$ or $\{14,23\}$, by lemma \ref{lem:types}. This means that leaves 1 and 2 are separated already at step $s=4$. They will remain separated ever after. In particular, we will have
    \begin{equation}
    \underline{1}^s\in r_2^s,\qquad\underline{2}^s\in r_1^s, \qquad \forall s\ge4.
    \end{equation}
    The fracture is
    \begin{equation}
    \dd(\underline{1}^3,w^3)=|f_{23}|=|f_{24}|+|f_{34}|,\qquad \text{if\ }f_{34}<0
    \end{equation}
    or
    \begin{equation}
    \dd(\underline{2}^3,w^3)=|f_{13}|=|f_{14}|+|f_{34}|,\qquad \text{if\ }f_{34}>0.
    \end{equation}

\begin{figure}[h]
 \centering
  \includegraphics[width=14cm,keepaspectratio]{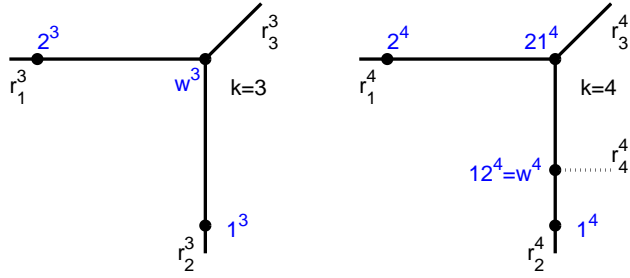}\\
  \caption{Modification and fracture occurring at step $s=4$, when $f_{34}<0$.}
  \label{fig:04}
  \end{figure}

In the previous two steps ($s=3$ or $s=4$) two entries in the last column of $F^s$ became $s$--active, while  one entry of $F^{s-1}$ stopped being active, due to the fracture. Moreover, properties \ref{prop:some_in_row} to  \ref{prop:sum}  in p. \pageref{prop:some_in_row} hold true.
\end{itemize}

\m
Assuming that properties \ref{prop:some_in_row} to  \ref{prop:sum} hold at step $(s-1)$, notice that exactly one fracture of one $(s-1)$--active entry of $F^{s-1}$ occurs at step $s$, for each $5\le s\le n$. Indeed, recall  (\ref{eqn:ineq}) and
 consider $i\in [s-1]$ ($i$ depending on $s$) such that
 \begin{equation}
 |f_{is}|=\min_{k\in [s-1]}|f_{ks}|.
 \end{equation}
 By genericity of $F$, such an index $i$ is unique and thus, some $(s-1)$--active entry on the $i$--th row of $F^s$ fractures.
 We have only one fracture at step $s$,  due to properties
\ref{prop:some_in_row} to \ref{prop:sum} and the fact that equalities (\ref{eqn:addit}) are not independent, for a fixed $s$.

\m
Now we proceed to describe $L$ as a tree, based on data in $F$. Assume, by recursion, that we have described the tree $L^{n-1}$  and that $L^{n-1}$ is trivalent.
Write $L'$ instead of $L^{n-1}$, for simplicity (similar meaning  for  $p'$, $q'$, $F'$, etc.).
Being trivalent, $L'$ is described by a finite family of bipartitions of $[n-1]$: 
\begin{equation*}
\{S_1,S_1^c\},\{S_2,S_2^c\},\ldots,\{S_t,S_t^c\},
\end{equation*}
where $t=n-4$  
is the number of inner edges of $L'$ (by recursion),  $S_j\subset[n-1]$, with $\card S_j\ge2$ and $\card S_j^c\ge2$ (by trivalency).  Moreover, the distances between pairs of adjacent vertices in $L'$ and the distances $\dd(p',p'q')$, $\dd(q'p',q')$ and $\dd(p',q')$ are certain entries of $|F'|$.
Now, the tree $L$ is a \emph{tropical modification}   of $L'$. That means that a ray $r_n^n$ sprouts up from  $L'$ at some point  of $L'$,  labeled $w$ temporarily, with the \emph{balancing condition} holding at $w$ inside $L$. The point $w$ becomes a vertex  of $L$ (although, it is not a vertex  in $L'$). By genericity, we face two cases:
\begin{enumerate}
\item \label{case:1} \emph{If $w$ belongs to the \emph{relative interior} of some inner edge $r$ of $L'$.} Say this segment corresponds to the bipartition $\{S_t,S_t^c\}$. We know that the leaves 1 and 2 are separated since step $s=4$, so that
    $$\{1,2\}\cap S_t \neq\emptyset\ \text{and}\ \{1,2\}\cap S_t^c\neq\emptyset.$$
    Say $1\in S_t$  and $2\in S_t^c$.  Removal of the relative interior of $r$
    splits the tree $L'$ into two subtrees, $L'_1$ and $L'_2$, named so that 1 is a leaf in $L'_1$.
    Then, the tree $L$ is described by
\begin{equation*}
\{\widehat{S_1},\widehat{S_1^c}\},\ldots,\{\widehat{S_{t-1}},\widehat{S_{t-1}^c}\},\{S_t\cup\{n\},S_t^c\},\{S_t,S_t^c\cup\{n\}\},
\end{equation*}
where
$$\widehat{S}=\begin{cases}
S\cup\{n\},&\text{ if $S^c$ is a subset of leaves of $L'_1$ or of $L'_2$,}\\S,&\text{otherwise.}
\end{cases}$$
Moreover, we know that the endpoints of $r$ are vertices of $L'$: let us label them $v_1, v_2$ temporarily, so that   $v_1\in L'_1$. Then
    \begin{equation*}
    \dd(v_1,v_2)=|f_{kl}|,
    \end{equation*}
    for some  $1\le k<l\le n-1$ and so  $f_{kl}$ is $(n-1)$--active. Due to tropical modification,  this  entry fractures, yielding
    \begin{equation*}
    |f_{kl}|=|f_{kn}|+|f_{ln}|
    \end{equation*}
    and so
    \begin{equation}\label{eqn:disj_11}
    \dd(v_1,w)=|f_{ln}|,\quad \dd(v_2,w)=|f_{kn}|,
    \end{equation}
    or
\begin{equation}\label{eqn:disj_12}
    \dd(v_1,w)=|f_{kn}|,\quad \dd(v_2,w)=|f_{ln}|.
    \end{equation}
    We decide  between (\ref{eqn:disj_11}) and (\ref{eqn:disj_12}) by computing the coordinates of $w$ in two different ways: beginning from $\underline{1}$ and beginning from $\underline{2}$.

\item \label{case:2} \emph{If $w$ belongs to the \emph{relative interior} of a ray $r_j'$, some $j\in[n-1]$.} Then $L$ is given by
\begin{equation*}
\{\{j,n\},
\{1,\ldots,j-1,j+1,\ldots,n-1\}\},
\{\widehat{S_1},\widehat{S_1^c}\},\{\widehat{S_2},\widehat{S_2^c}\},\ldots,\{\widehat{S_t},\widehat{S_t^c}\},
\end{equation*}
where
$$\widehat{S}=\begin{cases}
S\cup\{n\},&\text{ if $j\in S$,}\\S,&\text{otherwise.}
\end{cases}$$
Due to tropical modification,    one fracture of one $(n-1)$--active $f_{kl}$ occurs:
$$|f_{kl}|=|f_{kn}|+|f_{ln}|.$$ By recursion, we have
$|f_{kl}|=\dd(\underline{1}',\underline{12}')$ or $|f_{sl}|=\dd(\underline{2}',\underline{21}')$, and
recalling that  $\underline{1}'\in r_2'$ and $\underline{2}'\in r_1'$ (this holds true  since step $s=4$), we get
\begin{equation}\label{eqn:j}
j=2\text{\ or\ }j=1.
\end{equation}
\begin{itemize}
\item If $|f_{kl}|=\dd(\underline{1}',\underline{12}')$, then  $j=2$. We relabel $w$ as \underline{12}, relabel $\underline{12}'$ as $v$ and obtain
\begin{equation}\label{eqn:disj_21}
    \dd(\underline{1},\underline{12})=|f_{ln}|,\quad \dd(\underline{12},v)=|f_{kn}|,
    \end{equation}
    or
\begin{equation}\label{eqn:disj_22}
    \dd(\underline{1},\underline{12})=|f_{kn}|,\quad \dd(\underline{12},v)=|f_{ln}|.
    \end{equation}
    We decide  between (\ref{eqn:disj_21}) and (\ref{eqn:disj_22}) by computing the coordinates of $w$ in two different ways: beginning from $\underline{1}$ and beginning from $\underline{2}$.
\item If $|f_{kl}|=\dd(\underline{2}',\underline{21}')$, then the result is similar.
    \end{itemize}
\end{enumerate}

\m
If the couple $p,q$ is not generic, a sufficiently small perturbation $\tilde p,\tilde q$ of them is generic. We apply the previous paragraphs to $\tilde p,\tilde q$ and we obtain a line $\tilde L$. Then,  the line $L$ can be viewed as the result of the collapsing of some adjacent vertices on  $\tilde L$, or the points $p$  and $pq$ may coincide. Same for $q$ and $qp$. Passing from $\tilde L$ to $L$  amounts to vanishing of some $s$--active $\widetilde{f}_{kl}$,  with $1\le k<l\le s\le n$. The tree $L$ is caterpillar, though it might not be trivalent.
\end{proof}

%



\begin{figure}[h]
 \centering
  \includegraphics[keepaspectratio,width=10cm]{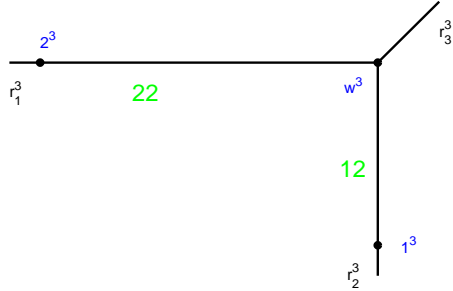}\\
  \caption{Construction of the tree $L$ in example \ref{ex:2}:  step $s=3$.}
  \label{fig:05}
\end{figure}
\begin{figure}[h]
 \centering
  \includegraphics[keepaspectratio,width=10cm]{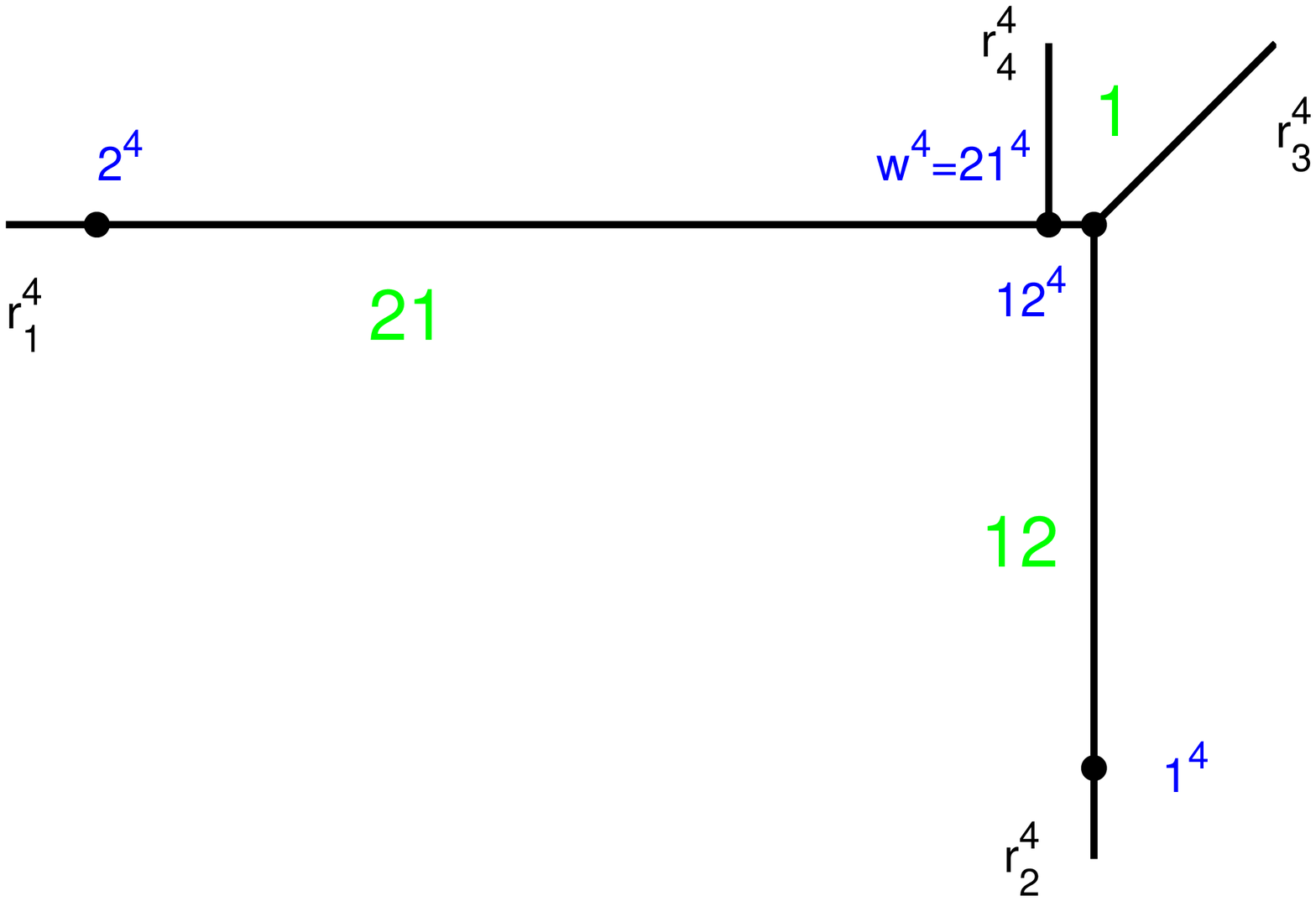}\\
  \caption{Construction of the tree $L$ in example \ref{ex:2}:  step $s=4$.}
  \label{fig:06}
\end{figure}
\begin{figure}[h]
 \centering
  \includegraphics[keepaspectratio,width=10cm]{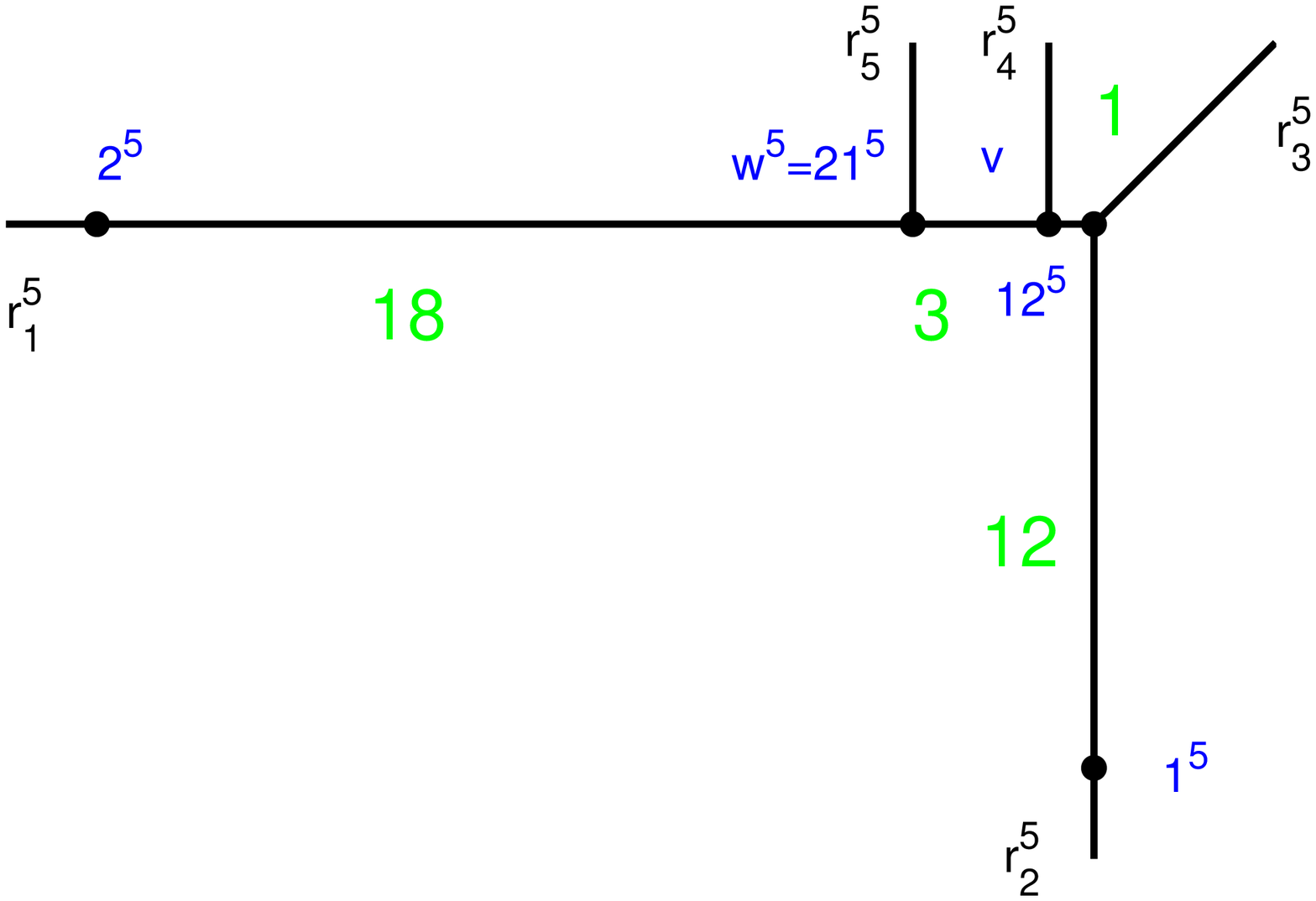}\\
  \caption{Construction of the tree $L$ in example \ref{ex:2}:  step $s=5$.}
  \label{fig:07}
\end{figure}
\begin{figure}[h]
 \centering
  \includegraphics[keepaspectratio,width=10cm]{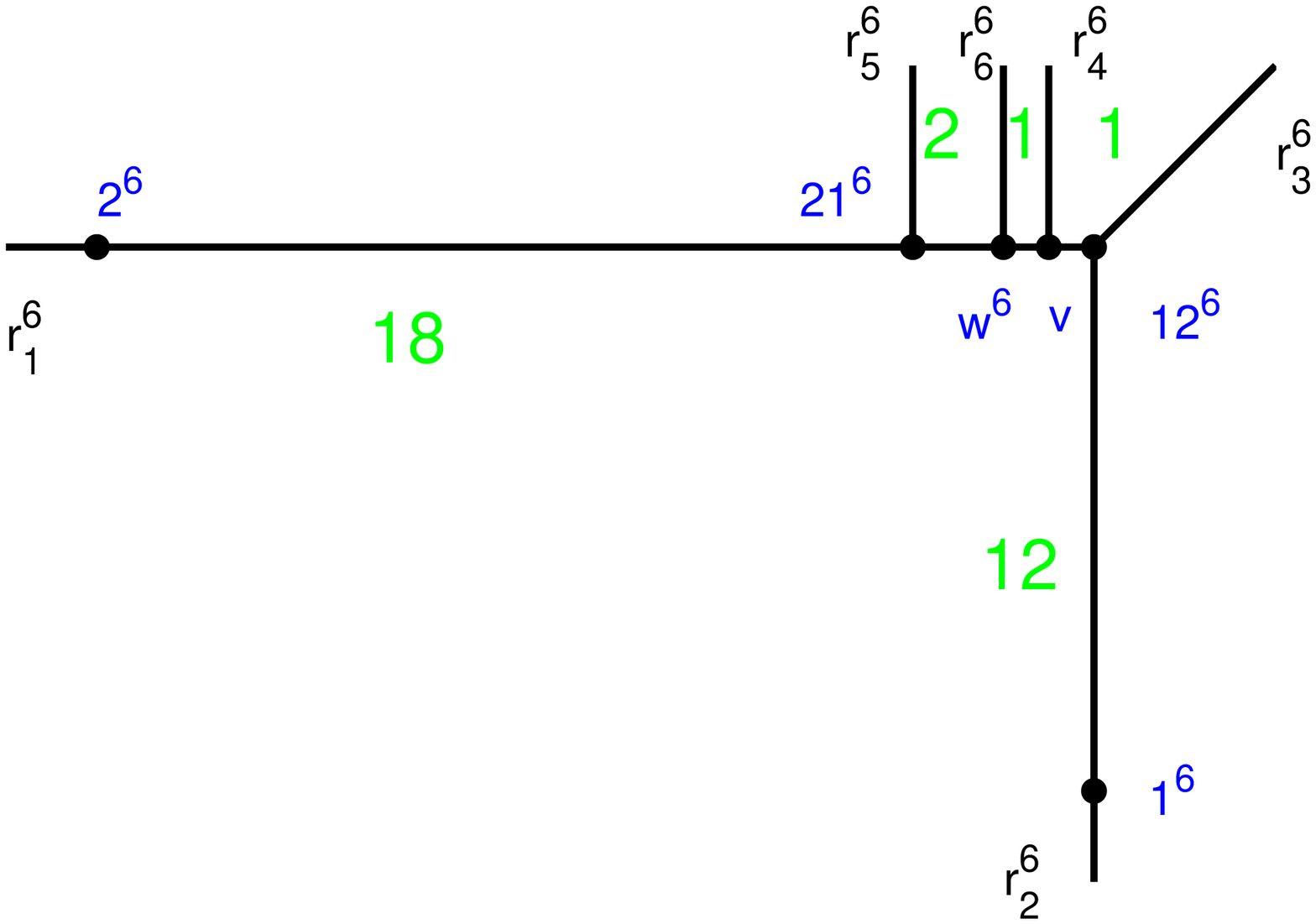}\\
  \caption{Construction of the tree $L$ in example \ref{ex:2}:  step $s=6$.}
\label{fig:08}
\end{figure}
\begin{figure}[h]
  \centering
  \includegraphics[keepaspectratio,width=10cm]{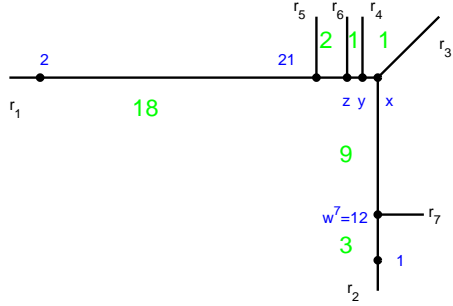}\\
  \caption{Construction of the tree $L$ in example \ref{ex:2}: final step.}
  \label{fig:09}
  \end{figure}

\begin{ex}\label{ex:2}  For $n=7$,
\begin{equation*}
\left[\begin{array}{rr}
0&-19\\-15&0\\-17&-14\\-16&-14\\-20&-21\\-18&-17\\-27&-15
\end{array}\right]
\end{equation*}
are the first two columns of a NI matrix $A=(a_{ij})$ (take, for instance $-28\le a_{st}\le -14$, if $s\neq t$ and $3\le t\le 7$).
Then $\dd(\underline{1},\underline{2})=\dd(\underline{1}^2,\underline{2}^2)=|f_{12}|=f_{12}=34$, by lemma \ref{lem:dist_given} and
\begin{equation*}\label{eqn:F}
F=\left(\begin{array}{rrrrrr}
34&22&21&18&20&31\\
&-12&-13&-16&-14&-3\\
&&-1&-4&-2&9\\
&&&-3&-1&10\\
&&&&2&13\\
&&&&&11
\end{array}\right).
\end{equation*}

For $3\le s\le 7$, active entries of $F^s$ will be boxed.

\begin{itemize}
\item $s=3$ (see figure \ref{fig:05}). 
The vertex  of the line $L^3$, denoted $w^3$,  is $[-a_{31},-a_{32},0]^t=[17,14,0]^t$, by Cramer's rule (\ref{eqn:Cramer}). We have a fracture
\begin{align*}
34&=22+12\\
|f_{12}|&=|f_{13}|+|f_{23}|
\end{align*}
and
\begin{equation}\label{eqn:w31}
\underline{1}^3 +12 \overline{e_2}=\colt{0}{-3}{-17}=\colt{3}{0}{-14}=\underline{2}^3+22 \overline{e_1}=w^3.
\end{equation}
 Thus
\begin{equation*}
\dd(\underline{1}^3,w^3)=12=|f_{23}|,\quad \dd(\underline{2}^3,w^3)=22=|f_{13}|.
\end{equation*}

The 3--active $f_{sl}$ are boxed below: \begin{equation*}
F=\left(\begin{array}{rrrrrr}
34&\boxed{22}&21&18&20&31\\
&\boxed{-12}&-13&-16&-14&-3\\
&&-1&-4&-2&9\\
&&&-3&-1&10\\
&&&&2&13\\
&&&&&11
\end{array}\right).
\end{equation*}

\item $s=4$ (see figure \ref{fig:06}). We have $-1=f_{34}<0$ so, by the remark after lemma  \ref{lem:types},   the type of $L^4$ is $\{14,23\}$. This means that $r_4^4$ and $r_1^4$ meet, and  $r_2^4$ and $r_3^4$ meet too inside $L^4$.
    This is case \ref{case:2} of the previous proof, with $j=2$. Since   $\underline{2}^4\in r_1^4$, then the point where $r_4^4$ and $r_1^4$ meet must be $\underline{21}^4$. The entry $f_{13}$ is 3--active and we have the fracture
    \begin{align*}
     22&=21+1\\
    |f_{13}|&=|f_{14}|+|f_{34}|.
    \end{align*}
     of $\dd(\underline{2}^3,w^3)=|f_{13}|$. In fact, $w^4$ can be relabeled  $\underline{21}^4$ and
     \begin{equation*}
    \underline{12}^4=\underline{1}^4+12\overline{e_2}=\colur{0}{-3}{-17}{-16},
    \quad \dd(\underline{1}^4,\underline{12}^4)=|f_{23}|=12,
    \end{equation*}
    \begin{equation*}
    \underline{21}^4=\underline{2}^4+21 \overline{e_1}=\colur{2}{0}{-14}{-14},\quad \dd(\underline{2}^4,\underline{21}^4)=|f_{14}|=21,
    \end{equation*}
    \begin{equation*}
    \underline{21}^4-\underline{12}^4= \colur{2}{3}{3}{2}= \colur{0}{1}{1}{0},\quad \dd(\underline{12}^4,\underline{21}^4)=1=|f_{34}|,
    \end{equation*}
    $$34=21+1+12$$
    and
\begin{equation*}
F=\left(\begin{array}{rrrrrr}
34&22&\boxed{21}&18&20&31\\
&\boxed{-12}&-13&-16&-14&-3\\
&&\boxed{-1}&-4&-2&9\\
&&&-3&-1&10\\
&&&&2&13\\
&&&&&11
\end{array}\right).
\end{equation*}

\item $s=5$ (see figure \ref{fig:07}). Then $|f_{14}|>|f_{15}|$ and $f_{14}$ is 4--active, so that
\begin{align*}
21&=18+3\\
|f_{14}|&=|f_{15}|+|f_{45}|
\end{align*}
is a fracture of $\dd(\underline{2}^4,\underline{21}^4)=|f_{14}|$. This is case \ref{case:2} of previous proof with $j=1$. Thus, the tree $L^5$ is given by
\begin{equation*}
\{15,234\},\{145,23\}
\end{equation*}
and it is caterpillar. We have
\begin{equation}\label{eqn:w51}
\underline{1}^5+12 \overline{e_2} +1\overline{e_{23}} +3\overline{e_{234}} =\colc{0}{1}{-13}{-13}{-20}=\colc{-1}{0}{-14}{-14}{-21}=\underline{2}^5+18 \overline{e_1},
\end{equation}
so that this point is $w^5$. Then,
\begin{equation*}
\dd(\underline{1}^5,\underline{12}^5)=|f_{23}|=12,\quad \dd(\underline{2}^5,\underline{21}^5)=|f_{15}|=18.
 \end{equation*}
In addition to $\underline{1}^5,\underline{12}^5,\underline{21}^5$ and $\underline{2}^5$, there is one more vertex  in $L^5$, denoted $v$ temporarily, and
we have
\begin{equation*}
\dd(\underline{12}^5,v)=|f_{34}|=1,\quad \dd(\underline{21}^5,v)=|f_{45}|=3,
\end{equation*}
$$34=18+3+1+12,$$

\begin{equation*}
F=\left(\begin{array}{rrrrrr}
34&22&21&\boxed{18}&20&31\\
&\boxed{-12}&-13&-16&-14&-3\\
&&\boxed{-1}&-4&-2&9\\
&&&\boxed{-3}&-1&10\\
&&&&2&13\\
&&&&&11
\end{array}\right).
\end{equation*}

\item $s=6$ (see figure \ref{fig:08}). Then $|f_{45}|>|f_{46}|$ and $f_{45}$ is 5--active, so that
\begin{align*}
3&=1+2\\
|f_{45}|&=|f_{46}|+|f_{56}|
\end{align*}
 is a fracture of $\dd(\underline{21}^5,v)=|f_{45}|$. A ray $r_6^6$ sprouts up from  the segment of  $L^5$ joining $\underline{21}^5$ and $v$. This is case \ref{case:1} of the previous proof. This happens at a point denoted $w^6$ temporarily and, therefore, tree $L^6$ is given by
 \begin{equation*}
 \{15,2346\},\{156,234\},\{1456,23\}.
 \end{equation*}
 Thus, $L^6$ is caterpillar and we have
 \begin{equation*}
 \underline{1}^6 +12 \overline{e_{2}}+1\overline{e_{23}}+1\overline{e_{234}}=\cols{0}{-1}{-15}{-15}{-20}{-18}=\cols{1}{0}{-14}{-14}{-19}{-17}=\underline{2}^6 +18 \overline{e_{1}}+2\overline{e_{15}},
 \end{equation*}
 and this point is $w^6$. Thus,
\begin{equation*}
 \dd(\underline{21}^6,{w}^6)=|f_{56}|=2,\quad\dd({w}^6,v)=|f_{46}|=1, \text{(this information is new)}
 \end{equation*}
 \begin{equation*}
 \dd(\underline{2}^6,\underline{21}^6)=|f_{15}|=18,\quad \dd(v,\underline{12}^6)=|f_{34}|=1,\quad \dd(\underline{12}^6,\underline{1}^6)=|f_{23}|=12,
 \end{equation*}
 $$34=18+2+1+1+12,$$
 \begin{equation*}
F=\left(\begin{array}{rrrrrr}
34&22&21&\boxed{18}&20&31\\
&\boxed{-12}&-13&-16&-14&-3\\
&&\boxed{-1}&-4&-2&9\\
&&&-3&\boxed{-1}&10\\
&&&&\boxed{2}&13\\
&&&&&11
\end{array}\right).
\end{equation*}

\item $s=7$ (see figure \ref{fig:09}). Then $|f_{23}|>|f_{27}|$ and $f_{23}$ is 6--active, whence
\begin{align*}
12&=3+9\\
|f_{23}|&=|f_{27}|+|f_{37}|
\end{align*}
is a fracture of $\dd(\underline{1}^6,\underline{12}^6)=|f_{23}|$. A ray $r_7^7$ sprouts out of $r_1^6$ (this is case \ref{case:2} of previous proof, with $j=2$) at a point labeled $w^7$. The tree $L=L^7$ is given by
\begin{equation*}
 \{15,23467\},\{156,2347\},\{1456,237\},\{13456,27\}
 \end{equation*}
 and we have
 \begin{equation}\label{eqn:w71}
 \underline{1}+3\overline{e_{2}}=\cole{0}{-12}{-17}{-16}{-20}{-18}{-27}=\cole{12}{0}{-5}{-4}{-8}{-6}{-15}
 =\underline{2}+18\overline{e_{1}}+2\overline{e_{15}}+1\overline{e_{156}}+1\overline{e_{1456}}+9\overline{e_{13456}}
 \end{equation}
 so that this point is $w^7$.
 Now, we relabel $w^7$ as $\underline{12}$.
In addition to vertices \underline{12} and \underline{21}, there are three more vertices in $L$, labeled $x,y$ and $z$.
We have
\begin{equation*}
\dd(\underline{1},\underline{12})=3,\quad \dd(\underline{12},x)=9,  \text{(this information is new)}
\end{equation*}
\begin{equation*}
\dd(x,y)=\dd(y,z)=1,\quad \dd(z,\underline{21})=2,\quad \dd(\underline{21},\underline{2})=18,
\end{equation*}
$$34=18+2+1+1+9+3,$$
\begin{equation*}
F=\left(\begin{array}{rrrrrr}
34&22&21&\boxed{18}&20&31\\
&-12&-13&-16&-14&\boxed{-3}\\
&&\boxed{-1}&-4&-2&\boxed{9}\\
&&&-3&\boxed{-1}&10\\
&&&&\boxed{2}&13\\
&&&&&11
\end{array}\right),
\end{equation*}
and finally
\begin{equation*}
|F|=\left(\begin{array}{ccccccc}
*&*&*&{\dd(\underline{2},\underline{21})}&*&*\\
&*&*&*&*&{\dd(\underline{1},\underline{12})}\\
&&{\dd(x,y)}&*&*&{\dd(\underline{12},x)}\\
&&&*&{\dd(y,z)}&*\\
&&&&{\dd(z,\underline{21})}&*\\
&&&&&*
\end{array}\right).
\end{equation*}

\end{itemize}
\end{ex}

\m
Remark: an algorithm is implicit in the the previous process; the details of it are postponed to a future paper.

\section*{Acknowledgement} We are grateful to an anonymous referee, who read an earlier version of this paper and pointed out some ways to improve it.


\centerline{\footnotesize{M. J. de la Puente. Dpto. de Algebra. Facultad de Matem\'{a}ticas. Universidad Complutense. Madrid. Spain. \texttt{mpuente@mat.ucm.es}}}
\end{document}